\documentclass[12pt, reqno]{amsart}

\usepackage{amsmath,amsfonts,amsbsy,amsgen,amscd,mathrsfs,amssymb,amsthm}
\usepackage{enumerate}

\usepackage[usenames,dvipsnames]{xcolor}
\usepackage[colorlinks=true,citecolor=blue,linkcolor=blue]{hyperref}
\usepackage[a4paper,margin=1in]{geometry}
\usepackage{setspace}
\usepackage{amsmath}

\allowdisplaybreaks[4]
\newtheorem{theorem}{Theorem}[section]
\newtheorem{definition}[theorem]{Definition}
\newtheorem{lemma}[theorem]{Lemma}
\newtheorem{corollary}[theorem]{Corollary}
\newtheorem{remark}[theorem]{Remark}
\newtheorem{proposition}[theorem]{Proposition}

\begin{document}
\begin{sloppypar}

\title[general $L_p$ $\mu$-projection body and general $L_p$ $\mu$-centroid body]{Extreme inequalities of general $L_p$ $\mu$-projection body and general $L_p$ $\mu$-centroid body}

\author{Chao Li}
\address{School of Mathematics and Statistics, Ningxia University, Yinchuan, Ningxia, 750021, China}
\email{lichao@nxu.edu.cn}

\author{Gangyi Chen}
\address{Gangyi Chen \newline \indent School of Mathematics and Information Science, Guangzhou University, Guangzhou, 510006, China}
\thanks{Corresponding author: Gangyi Chen} 
\email{gychen@gzhu.edu.cn}

\begin{abstract}
In this paper, we introduce the concept of general $L_p$ projection body and general $L_p$ centroid body of general measures with positive homogeneity density function, and prove the corresponding extreme inequalities. Meanwhile, we also study their measure comparison problem and monotone inequalities.
\end{abstract}


\maketitle

\thispagestyle{empty}

\section{Introduction and main results}
 Let $\mathscr{K}^n$ denote the set of convex bodies (compact, convex subsets with non-empty interiors) in Euclidean space $\mathbb{R}^n$. Let $\mathscr{K}^n_o$ denote the set of convex bodies containing the origin in their interiors in $\mathbb{R}^n$. For the set of star bodies (about the origin) in $\mathbb{R}^n$, we write $\mathscr{S}^n_o$. Besides, $\mathbb{S}^{n-1}$ denotes the unit sphere and $V_n(K)$ denotes the $n$-dimensional volume of the body $K$. The volume of the Euclidean unit ball $B^n$ is given by
$\kappa_n=\pi^{n/2}/\Gamma(1+n/2)$.

As is known to all , the projection body has always been one of the main research contents of Brunn-Minkowski theory, which can be described as follows: Let $K\in\mathscr{K}^n$, the projection body $\Pi K$ of $K$ is an origin-symmetric convex body whose support function is defined by (see \cite{SRAA})
\begin{align*}
h({\Pi K},u)=V_{n-1}(K | u^\bot), \qquad u \in \mathbb{S}^{n-1},
\end{align*}
where $V_{n-1}$ is the ($n-1$)-dimensional volume and $K | u^\bot$ represents the image of the orthogonal projection of $K$ onto the subspace orthogonal to $u$.

Since the concept of projection body was proposed by Minkowski, it has attracted extensive attention of many scholars studying convex geometry. Moreover, the Shephard problem \cite{SHEP} of projection body, Petty (affine) projection inequality \cite{PETCM,PETTY}, Zhang's projection inequality \cite{ZANGY}, Schneider conjecture of projection inequality \cite{SCHRF} and a series of far-reaching and meaningful research topics of projection body are well known to scholars in convex geometry. In recent years, the concept of projection body has been extended to $L_p$ projection body \cite{LDYZG}, nonsymmetric $L_p$ projection body \cite{HSFE}, Orlicz projection body \cite{LTYZX} and complex projection body \cite{AJBD}, which further enriches the research content of projective body.

The research of centroid body, which has the same status as projection body, has also attracted the attention of scholars. For $K\in\mathscr{S}^n_o$, the centroid body $\Gamma K$ of $K$ is a convex body,  whose support function is defined by (see \cite{SRAA})
\begin{align*}
h(\Gamma K,u)=\frac{1}{V(K)}\int_K |u \cdot x | dx, \qquad  u \in \mathbb{S}^{n-1}.
\end{align*}
In particular, the notion of centroid body is also extended to $L_p$-centroid body \cite{LZWK}, general $L_p$-centroid body \cite{HSFE,LMIG}, Orlicz centroid body \cite{LHGEZ}, Orlicz-Lorentz centroid body \cite{NVHV} and complex centroid body \cite{HZER}. Moreover, the study of Busemann-Petty centroid inequality \cite{PETTA} and Shephard type problem \cite{LUCRA,SCHRG} of centroid body has had a far-reaching impact, see \cite{CSGP,FMMT,CYZA,LCWD,ZHANG,ZZGX}. We can also refer to Gardner \cite{GRJD} and Schneider's \cite{SRAA} two golden classics, which contain a relatively complete collection of a large number of research results on projection body and centroid body.

In 2014, Milman and Rotem \cite{MR} studied the complementary Brunn-Minkowski inequality and isoperimetric inequality of homogeneous and non-homogeneous measures. Afterwards, Livshyts \cite{LIV} proposed the concept of mixed $\mu$ measure and put forward a new notion which relates projections of convex bodies to a given measure $\mu$, which is a generalization of the Lebesgue volume of a projection.

\noindent{\bf Definition 1.A (Page 5 of \cite{LIV}).}~~{\it Suppose $K\in\mathscr{K}^n$  and $\mu$ is a measure on $\mathbb{R}^{n}$ with continuous density $\omega$. Consider a unit vector $u \in \mathbb{S}^{n-1}$. The function on the cylinder $\mathbb{S}^{n-1} \times[0,1]$ is defined by
\begin{align*}
p_{\mu, K}( u , t):=\frac{n}{2} \int_{\mathbb{S}^{n-1}}|u \cdot v| d S_{\mu}(tK, v).
\end{align*}
Then the $\mu$-projection function on the unit sphere is defined by
\begin{align*}
P_{\mu, K}( u ):=\int_{0}^{1} p_{\mu, K}(u,t)dt,
\end{align*}
where $dS_{\mu}(K,\cdot)$ denote the surface $\mu$-area measure.}

 In the particular case of Lebesgue measure $\mathcal{L}$, $P_{\mathcal{L}, K}( u )=V_{n-1}(K| u ^{\perp})$.

Afterwords, Wu \cite{WDH2} extended the definition of  $\mu$-projection function  \cite{LIV} to $L_p$ cases and studied the Firey-Shephard problem for homogeneous measure. Recently, Langhurst, Roysdon and Zvavitch \cite{LRZ} proved Zhang's inequality with respect to $h(\Pi_\mu K,u)=\frac{1}{n}p_{\mu,K}(u,1)$, for $u \in \mathbb{S}^{n-1}$. 

In 2022, Wu and Bu \cite{WDH3} introduced the following concept of the $L_{p}$-centroid body for general measures and studied the measure-comparison problem.

\noindent{\bf Definition 1.B (Page 3 of \cite{WDH3}).}~~{\it Let $p \geq 1$ and $\mu$ be a measure on $\mathbb{R}^{n}$ with continuous density $\omega$. For $K\in\mathscr{S}^n_o$, the $(p, \mu)$-centroid body $\Gamma_{\mu, p} K$ of $K$ is the body defined by, 
\begin{align*}
h(\Gamma_{\mu,p} K,x)^p=\frac{1}{\mu(K)} \int_{K}|x \cdot y|^{p} d \mu(y),
\end{align*}
for every $x \in \mathbb{R}^{n}$.}

In the case of Lebesgue measure $\mathcal{L}, \Gamma_{\mathcal{L}, p} K$ is just the classical $L_{p}$-centroid body $\Gamma_{p} K$.

In recent years, the research on extending some conclusions of the classic Brunn-Minkowski theory to arbitrary measures has attracted much attention. For more recent developments with respect to the convex body of arbitrary measures, also see \cite{AHQZ,HOS,KPZA,KLL,LCWG,NJY,ROYS,RX,ZVAV}.

In this paper, we apply the method of Haberl and Schuster \cite{HSFE}, and first put forward the following concept of the projection body of general measures to the asymmetric case.

Let  $p \geq 1$ and $\mu$ be a $1/s$ homogeneous Borel measure on $\mathbb{R}^n$ with $1/r$ homogeneous density $\omega$, where $r>0$ and $s$ satisfy the equation $s=\frac{1}{n+\frac{1}{r}}$. For $K \in \mathscr{K}_{\mathrm{o}}^n$ and function $\psi_{\tau}:
\mathbb{R} \rightarrow [0,\infty)$ defined by
$\psi_{\tau}(t)=|t|+\tau t$, $\tau \in [-1,1]$, the general $L_p$ $\mu$-projection body $\Pi_{\mu,p}^{\tau}K
\in \mathscr{K}_{\mathrm{o}}^n$ of $K$ is defined by (see Section \ref{SEC3}) 
\begin{align*}
h(\Pi_{\mu,p}^{\tau}K,u)^p=c_{n,p}(\tau)\int_{\mathbb{S}^{n-1}}\psi_{\tau}(u\cdot
v)^p dS_{\mu,p}(K,v),
\end{align*}
where $u \in \mathbb{S}^{n-1}$ and $dS_{\mu,p}(K, \cdot)$ denotes the $L_p$-surface $\mu$-area measure of $K$. In the case of Lebesgue measure $\mathcal{L}$, $\Pi_{\mathcal{L}, p}^\tau K$ is just the general $L_{p}$-projection body $\Pi_{p}^\tau K$ defined by Haberl and Schuster \cite{HSFE}.

Suppose $\mu$ is a $1/s$ homogeneous Borel measure on $\mathbb{R}^n$ with $1/r$ homogeneous density $\omega$, where $r>0$ and $s$ satisfy the equation $s=\frac{1}{n+\frac{1}{r}}$. Let  $p \geq 1$ and $K \in \mathscr{S}^n_o$, for $u \in \mathbb{S}^{n-1}$, the general $L_p$ $\mu$-centroid body $\Gamma_{\mu,p}^{\tau}K\in \mathscr{K}^n_{\mathrm{o}}$ of $K$ is defined by (see Section \ref{SEC3}) 
\begin{align*}
h(\Gamma_{\mu,p}^{\tau}K,u)^p=&\frac{2}{\alpha_{n,p}(\tau)\mu(K)}\int_K \psi_{\tau}(u\cdot
x)^pd\mu(x).
\end{align*}
In particular, if $\mu$ is the Lebesgue measure $\mathcal{L}$, then $\Gamma_{\mathcal{L}, p}^\tau K$ is  the general $L_p$-centroid body $\Gamma_{p}^\tau K$ defined by Feng, Wang and Lu \cite{FWLF}. In addition, $[(\alpha_{n,p}(\tau)V(K))/(2c_{n,p}(\tau))]^\frac{1}{p}\Gamma_{\mathcal{L}, p}^\tau K$ is the general $L_p$-moment body $M^\tau_{p}K$ (see Page 10 of \cite{HSFE}). 

The following are the main results of this paper. Firstly, we study the  extreme inequality of the polar body of general $L_p$ $\mu$-projection body.

\begin{theorem}\label{TLD1}
Let $\mu$ be a $1/s$ homogeneous Borel measure on $\mathbb{R}^n$ with $1/r$ homogeneous density function $\omega$, and $\omega$ is even,  where $r>0$ and $s$ satisfy the equation $s=\frac{1}{n+\frac{1}{r}}$, if $K\in \mathscr{K}^n_o$, $p\geq 1$, and $\tau \in [-1,1]$, then
\begin{align}\label{TYTJZ}
\mu(\Pi_{\mu,p}^\ast K) \leq \mu(\Pi_{\mu,p}^{\tau,\ast} K) \leq \mu(\Pi_{\mu,p}^{\pm,\ast} K),
\end{align}
when $\tau \neq
0$, equality holds in the left inequality if and only if $\Pi_{\mu,p}^{\tau}K$ is origin-symmetric, and when $\tau \neq \pm 1$, equality holds in the right inequality if and only if $\Pi_{\mu,p}^{\pm} K$ is origin-symmetric.
\end{theorem}
\begin{remark}
If $K$ is origin-symmetric, then (\ref{TYTJZ}) is an identical equation.
\end{remark}

\begin{remark}
If $\mu$ is the Lebesgue measure, then Theorem \ref{TLD1} is proved by Haberl and Schuster (See Theorem 6.3 of \cite{HSFE}).
\end{remark}

Furthermore, the extreme inequality of the polar body of the  general $L_p$ $\mu$-centroid body is also proved.

\begin{theorem}\label{TLD2}
Let $\mu$ be a $1/s$ homogeneous Borel measure on $\mathbb{R}^n$ with $1/r$ homogeneous density function $\omega$, and $\omega$ is even, where $r>0$ and $s$ satisfy the equation $s=\frac{1}{n+\frac{1}{r}}$, if $K\in \mathscr{S}^n_o$, $p\geq 1$, and $\tau \in [-1,1]$, then
\begin{align}\label{ZXTJZ}
\mu(\Gamma_{\mu,p}^\ast K) \leq \mu(\Gamma_{\mu,p}^{\tau,\ast} K) \leq \mu(\Gamma_{\mu,p}^{\pm,\ast} K),
\end{align}
when $\tau \neq
0$, equality holds in the left inequality if and only if $\Gamma_{\mu,p}^{\tau}K$ is origin-symmetric, and when $\tau \neq
\pm 1$, equality holds in the right inequality if and only if $\Gamma_{\mu,p}^{\pm} K$ is origin-symmetric.
\end{theorem}
\begin{remark}
If $K$ is origin-symmetric, then (\ref{ZXTJZ}) is an identical equation.
\end{remark}

This paper is organized as follows: In Section \ref{SEC2}, we list some basic knowledge of the convex geometry. In Section \ref{SEC3}, we introduce the general $L_p$ $\mu$-projection body and the general $L_p$ $\mu$-centroid body, and study some of their basic properties. We also prove the measure comparison inequality of general $L_p$ $\mu$-projection body and general $L_p$ $\mu$-centroid body and their polar bodies. In Section \ref{SEC4}, We prove Theorem \ref{TLD1} and Theorem \ref{TLD2}. Finally, in Section \ref{SEC5}, we prove the Brunn-Minkowski type inequality of general $L_p$ $\mu$-projection body, and also obtain the monotone inequality of general $L_p$ $\mu$-centroid body.

\section{Preliminaries}\label{SEC2}
Below we will list some basic facts of convex geometry.

\subsection{$L_p$-Minkowski combination and $L_p$-radial combination}$\qquad$

A convex body $K\in \mathscr{K}^n$, is uniquely determined by its support function, $h(K, \cdot)=h_K:\mathbb{R}^n\rightarrow\mathbb{R}$, which is defined by
\begin{align}\label{SF}
h(K, x)=\max\{ x\cdot y : y\in K\}, \ \ \ \ x\in \mathbb{R}^n,
\end{align}
where $x\cdot y$ denotes the standard inner product of $x$ and $y$ in $\mathbb{R}^n$. It is also clear from the definition that $h(K,u) \leq h(L,u)$ for $u \in \mathbb{S}^{n-1}$,  if and only if $K\subseteq L$.

For $K,L\in\mathscr{K}_o^n$, $p\geq1$ and $\alpha,\beta\geq0$ (not both zero), the $L_p$-Minkowski combination, $\alpha\cdot K+_p\beta\cdot L$, of $K$ and $L$ is defined by (see \cite{SRAA})
\begin{align}\label{lpmz}
h(\alpha\cdot K+_p\beta\cdot L,\cdot)^p=\alpha h(K,\cdot)^p+\beta h(L,\cdot)^p,
\end{align}
where $``+_p"$ denotes the $L_p$-sum and $\alpha\cdot K=\alpha^{1 / p} K$ is the $L_p$-Minkowski scalar multiplication. If $p=1$,  (\ref{lpmz}) is just the classical Minkowski linear combination.

If $K$ is a compact star-shaped set (about the
origin) in $\mathbb{R}^n$, then its radial function,
$\rho_K=\rho(K,\cdot):\mathbb{R}^n\setminus\{0\}\rightarrow[0,\infty)$, is defined by
\begin{align}\label{RF}
\rho(K,x)=\max\{c\geq0: c x\in K\}, \ \ x\in\mathbb{R}^n\setminus \{0\}.
\end{align}
If $\rho_K$ is positive and continuous, $K$ will be called a star body (with respect to the origin). Obviously, if $\alpha>0$, then $\rho_{\alpha K}=\alpha\rho_K$.

By (2.3), we easily see that if $K\subseteq L$, then
\begin{align}\label{RC}
\rho(K,\cdot)\leq \rho(L,\cdot),
\end{align}
with equality if and only if $K=L$.

If $K\in\mathscr{K}^n_o$, the polar body $K^\ast$, of $K$ is defined by
\begin{align}\label{PB}
K^\ast=\{x\in\mathbb{R}^n:x\cdot y\leq 1,y\in K\}.
\end{align}

From (\ref{SF}), (\ref{RF}) and (\ref{PB}), it follows that if $K\in\mathscr{K}^n_o$, then
\begin{align}\label{SRF}
\rho(K^\ast,\cdot)=\frac{1}{h(K,\cdot)}, \quad \quad \quad h(K^*,\cdot)=\frac{1}{\rho(K,\cdot)}.
\end{align}

If $K,L \in \mathscr{S}^n_o$, $p\neq0$ and $\alpha,\beta\geq 0$ (not both zero), the $L_p$-radial combination, $\alpha\circ K\widetilde{+}_p\beta\circ L \in \mathscr{S}^n_o$, of $K$ and $L$ is defined by (see \cite{SRAA})
\begin{align}\label{RA}
\rho(\alpha\circ K\widetilde{+}_{p}\beta\circ L,\cdot)^p=\alpha\rho(K,\cdot)^p+\beta\rho(L,\cdot)^p,
\end{align}
where $\widetilde{+}_p$ is called the $L_p$-radial addition and $\alpha\circ K=\alpha^{1 / p} K$ is the $L_p$-radial scalar multiplication. If $p = 1$, then $\alpha\circ K\widetilde{+} \beta\circ L$ is the classical radial combination. If $p\geq0$, then $\alpha\circ K\widetilde{+}_{-p} \beta\circ L$ is called $L_p$-harmonic radial combination.

From (\ref{lpmz}), (\ref{SRF}) and (\ref{RA}), for $K,L\in\mathscr{K}_o^n$, $p\geq1$ and $\alpha,\beta\geq0$ (not both zero), then
\begin{align}\label{hrzh}
(\alpha\cdot K+_p\beta\cdot L)^*=\alpha\circ K^*\widetilde{+}_{-p} \beta\circ L^*.
\end{align}

\subsection{$L_p$ mixed $\mu$-measure }\label{Sect 3}$\qquad$

A measurable function $\omega: \mathbb{R}^{n} \rightarrow \mathbb{R}^{+}$ is called $1/r$-homogeneous, $r \in (-\infty, \infty) \backslash\{0\}$, if for $x \in \mathbb{R}^{n}$ and $\lambda>0$, then (see Definition 2.6 of \cite{MR})
\begin{equation*}
\omega(\lambda x)=\lambda^{1/r} \omega(x),
\end{equation*}
with the interpretation when $r=\infty$ that $\omega(\lambda x)=\omega(x)$ for all $\lambda>0$, i.e. that $\omega$ is constant along rays from the origin.

A Borel measure $\mu$ on $\mathbb{R}^{n}$ is called $1/s$-homogeneous, $s \in(-\infty, \infty) \backslash\{0\}$, if for a Borel subset $\mathcal{A}\subset \mathbb{R}^{n}$, and $\lambda>0$, then (see Definition 2.7 of \cite{MR})
\begin{equation*}
\mu(\lambda \mathcal{A})=\lambda^{1/s} \mu(\mathcal{A}),
\end{equation*}
with the interpretation when $s=\infty$ that $\mu(\lambda\mathcal{A})=\mu(\mathcal{A})$ for all $\lambda>0$.

For convenience, we call the measure $\mu$ satisfying the homogeneity condition as:

{\bf ({\rm I})} Let $\mu$ be a $1/s$ homogeneous Borel measure on $\mathbb{R}^n$ with $1/r$ homogeneous density $\omega$, where $r>0$ and $s$ satisfy the equation $s=\frac{1}{n+\frac{1}{r}}$.

A function $\omega:\mathbb{R}^n\rightarrow \mathbb{R}^+$ is called $r$-concave if for every $x,y\in \mathbb{R}^n$ such that $\omega(x)\omega(y)>0$, and for every $\alpha \in [0,1]$, one has
\begin{align*}
\omega(\alpha x+(1-\alpha)y)\geq (\alpha \omega(x)^r+(1-\alpha)\omega(y)^r)^{1/r}.
\end{align*}
We remark that $0$-concave functions are also called log-concave functions.

The following generalized Brunn-Minkowski inequality is well known (see \cite{BOC,BHJ}). Let $r\in [-\frac{1}{n},+\infty]$, and let $\mu$ be a measure on $\mathbb{R}^n$ with $r$-concave density $\omega$. Let 
\begin{align*}
\frac{1}{s}=\frac{1}{r}+n.
\end{align*}
Then the measure $\mu$ is $s$-concave on $\mathbb{R}^n$. That is, for every pair of Borel sets $\mathcal{A}$ and $\mathcal{B}$ and for every $\alpha\in[0,1]$ one has
\begin{align*}
\mu(\alpha \mathcal{A}+(1-\alpha) \mathcal{B})\geq( \alpha\mu(\mathcal{A})^s+(1-\alpha)\mu(\mathcal{B})^s)^{1/s}.
\end{align*}

For convenience, we call the measure $\mu$ satisfying the concavity condition as:

{\bf ({\rm II})} Let $\mu$ be a $s$-concave Borel measure on $\mathbb{R}^n$ with $r$-concave density $\omega$, where $r>0$ and $s$ satisfy the equation $s=\frac{1}{n+\frac{1}{r}}$.

 If $\mu$ satisfies the homogeneity condition {\bf ({\rm I})}, then
\begin{align*}
\mu(K)=\int_{K} d\mu(x),
\end{align*}
from the polar coordinate transformation, we get
\begin{align}\label{MUM}
\mu(K)=s\int_{\mathbb{S}^{n-1}}\rho(K,u)^\frac{1}{s}d\mu(u).
\end{align}
Let $\omega=1$ in (\ref{MUM}), then we get
\begin{align*}
V_{n}(K)=\frac{1}{n} \int_{\mathbb{S}^{n-1}} \rho(K,u)^n d u.
\end{align*}

In 2017, Wu \cite{WDH1} extended the mixed $\mu$-measure to $L_p$ case and proposed the following $L_{p}$ mixed $\mu$-measure: Let $K, L \in \mathscr{K}^n_o$, $p \geq 1$ and $\mu$ be a measure on $\mathbb{R}^{n}$, the $L_p$ mixed $\mu$-measure $\mu_{p}(K, L)$, of $K$ and $L$, is defined by
\begin{align*}
\mu_{p}(K, L)=\liminf _{\varepsilon \rightarrow 0} \frac{\mu(K+_p \varepsilon\cdot L)-\mu(K)}{\varepsilon} .
\end{align*}
Moreover, Wu \cite{WDH1} also got the following integral expression:
\begin{align}\label{lmhmm}
\mu_{p}(K, L)=\frac{1}{p} \int_{\mathbb{S}^{n-1}} h(L,u)^p d S_{\mu, p}(K, u).
\end{align}
Here, $dS_{\mu,p}(K, \cdot)$ denotes the $L_p$-surface $\mu$-area measure of $K$. 

Especially, if $K$ is a $C^{2}$-smooth and strictly convex body with the origin in its interior, then for all $u\in \mathbb{S}^{n-1}$, the $L_p$ surface $\mu$-area measure $S_{\mu, p}(K, u)$ has representation
\begin{align*}
d S_{\mu, p}(K, u)=h_{K}^{1-p}(u) f_{K}(u) \omega\left(\nabla h_{K}(u)\right) d u.
\end{align*}
where $f_{K}$ denotes the curvature function of $K$. 

In particular, if $p=1$, then $d S_{\mu, 1}(K, \cdot)=dS_{\mu}(K,\cdot)$ is the surface $\mu$-area measure defined by Livshyts \cite{LIV}, and its expression is
\begin{align*}
d S_{\mu}(K, u)=f_{K}(u) \omega\left(\nabla h_{K}(u)\right) d u.
\end{align*}
In addition, Wu \cite{WDH1} obtained that the measure $S_{\mu, p}(K, \cdot)$ is absolutely continuous with respect to $S_{\mu}(K, \cdot)$, and has Radon-Nikodym derivative,
\begin{align*}
\frac{d S_{\mu, p}(K, \cdot)}{dS_{\mu}(K, \cdot)}=h_{K}^{1-p} .
\end{align*}

Note that $\mu_p(K, L)$ is not homogeneous in $K$, since $\mu$ is not homogeneous. However, $\mu_p(K, L)$ is $p$-homogeneous in $L$ i.e., $\mu_{p}(K, \lambda L)=\lambda^{p} \mu_p(K, L)$, for $\lambda>0$.

Meanwhile, from $L_p$ mixed $\mu$-measure, Wu \cite{WDH1} defined the following concept of the $L_p$ mixed volume with respect to $\mu$: Let $p \geq 1$, $K, L \in \mathscr{K}^n_o$ and $\mu$ be a measure on $\mathbb{R}^{n}$, then the $L_p$ mixed volume with respect to the measure $\mu$, $V_{\mu, p}(K, L)$, of $K$ and $L$, is defined by
\begin{align}\label{lphmv}
V_{\mu, p}(K, L)=\int_{0}^{1} \mu_{p}(t\cdot K, L)dt.
\end{align}

 If $\mu$ satisfies the homogeneity condition {\bf ({\rm I})}, from (\ref{lmhmm}) and (\ref{lphmv}), Wu \cite{WDH1} obtained 
\begin{align}\label{lphmt}
V_{\mu, p}(K, L)=sp\mu_p(K, L)=s\int_{\mathbb{S}^{n-1}} h(L,u)^p d S_{\mu, p}(K, u).
\end{align}

Obviously, if $\mu$ satisfies the homogeneity condition {\bf ({\rm I})}, then 
\begin{align*}
V_{\mu, p}(K, K)=\mu(K), \qquad \mu_p(K, K)=\frac{1}{sp}\mu(K).
\end{align*}

 If $\mu$ satisfies the homogeneity condition {\bf ({\rm I})}, Wu \cite{WDH1} also deduced $\mu_{p}(\lambda K, L)=\lambda^\frac{1-sp}{s} \mu_{\mathrm{p}}(K, L)$, for $\lambda>0$, thus
\begin{align}\label{lphmt2}
d S_{\mu, p}(\lambda K, \cdot)=\lambda^\frac{1-sp}{s}d S_{\mu, p}(K, \cdot).
\end{align}

Clearly, in the case of Lebesgue measure $\mathcal{L}$, we have
\begin{align*}
V_{\mathcal{L},p}(K, L)=V_{p}(K, L).
\end{align*}

Note that, for $p=1, (\ref{lphmt})$ is just
\begin{align*}
\mu_{1}(K, L)=\int_{\mathbb{S}^{n-1}} h_{L}(u) d S_{\mu, 1}(K, u),
\end{align*}
which is given by Livshyts \cite{LIV}. 

In addition, Wu \cite{WDH1} also proved the following Minkowski inequality for $L_p$-mixed volume with respect to the measure $\mu$: Let $\mu$ satisfy the homogeneity condition {\bf ({\rm I})} and concavity condition {\bf ({\rm II})}, if $K,L \in \mathscr{K}^n_o$ and $p\geq 1$, then
\begin{align}\label{lphtmi}
\mu_{p}(K, L) \geq \frac{1}{s} \mu(K)^{1-s} \mu(L)^{s}+\frac{1-p}{s p} \mu(K),
\end{align}
and
\begin{align}\label{lphtm2}
V_{\mu, p}(K, L) \geq p \mu(K)^{1-s} \mu(L)^{s}+(1-p) \mu(K).
\end{align}

\begin{lemma}\label{LEMBM}
 Let $\mu$ satisfy the homogeneity condition {\bf ({\rm I})} and concavity condition {\bf ({\rm II})}. If $K,L \in \mathscr{K}^n_o$, $\alpha,\beta\geq 0$ and $p\geq 1$, then
\begin{align*}
\frac{1+(\alpha+ \beta)(p-1)}{p}\mu(\alpha\cdot K +_p \beta\cdot L)^s\geq \alpha\mu(K)^{s}+ \beta\mu(L)^{s}.
\end{align*}
\end{lemma}
\begin{proof}
For $p\geq 1$, by (\ref{lmhmm}), (\ref{lpmz}) and (\ref{lphtmi}), we know that for any $M \in \mathscr{K}_{o}^{n}$
\begin{align*}
\mu_p(M, \alpha\cdot K+_p \beta\cdot L)
=& \alpha \mu_p(M, K)+ \beta \mu_p(M, L) \\
\geq&\frac{1}{s} \mu(M)^{1-s} [\alpha\mu(K)^{s}+ \beta\mu(L)^{s}]+\frac{(\alpha+ \beta)(1-p)}{s p}\mu(M),
\end{align*}
Let $M=\alpha\cdot K +_p \beta\cdot L$ in the above formula, we can deduce
\begin{align*}
\frac{1+(\alpha+ \beta)(p-1)}{p}\mu(\alpha\cdot K +_p \beta\cdot L)^s\geq \alpha\mu(K)^{s}+ \beta\mu(L)^{s}.
\end{align*}
This completes the proof.
\end{proof}
In particular, if $\alpha=\beta=1$, Lemma \ref{LEMBM} is exactly Wu's result (see Page 171 of \cite{WDH1}).

\subsection{$L_p$ dual mixed $\mu$ measure }\label{Sect 3}

\begin{definition}[\cite{LCC}]\label{DEF1}
Let $K,L \in \mathscr{S}^n_o$, $p\neq0$ and $r>0$. If $\mu$ satisfies the homogeneity condition {\bf ({\rm I})}, then the $L_p$ dual mixed $\mu$ measure can be defined by
\begin{align*}
\widetilde{V}_{\mu,p}(K, L)=sp\lim _{\varepsilon \rightarrow 0^{+}} \frac{\mu\left(K\widetilde{+}_{p}\varepsilon \circ L\right)-\mu(K)}{\varepsilon}. 
\end{align*}
\end{definition}

Next, we will give the integral expression of the $L_p$ dual mixed $\mu$ measure. 
\begin{lemma}[\cite{LCC}]\label{Lpdmmm}
Let $K,L \in \mathscr{S}^n_o$ and $p\neq0$. If $\mu$ satisfies the homogeneity condition {\bf ({\rm I})}, then
\begin{align}\label{LpHT}
\widetilde{V}_{\mu,p}(K, L)=s\int_{\mathbb{S}^{n-1}}\rho(K,u)^\frac{1-sp}{s}\rho(L,u)^pd\mu(u).
\end{align}
\end{lemma}

Note that Lemma \ref{Lpdmmm} is a special case of Gardner, Hug, Weil, Xing and Ye (See (16) of \cite{GRDXD}). By Lemma \ref{Lpdmmm} and applying  H\"{o}lder's integral inequality and Minkowski's integral inequality, we will easily prove the following Minkowski inequality and Brunn-Minkowski inequality for $L_p$ dual mixed $\mu$ measure.

\begin{lemma}[\cite{LCC}]\label{DMMBS}
Let $\mu$ satisfy the homogeneity condition {\bf ({\rm I})} and $K,L \in \mathscr{S}^n_o$, if $0<p< n+\frac{1}{r}$, then
\begin{align}\label{mi}
\widetilde{V}_{\mu,p}(K, L)\leq \mu(K)^{1-sp}\mu(L)^{sp}.
\end{align}
If $p<0$ or $p> n+\frac{1}{r}$, then inequality (\ref{mi}) is reversed. Equality holds if and only if $K$ and $L$ are dilatates.
\end{lemma}

\begin{lemma}[\cite{LCC}]\label{rbmi}
Let $\mu$ satisfy the homogeneity condition {\bf ({\rm I})} and $K,L \in \mathscr{S}^n_o$, if $0<p< n+\frac{1}{r}$ and $\alpha,\beta>0$, then
\begin{align}\label{bmi}
\mu(\alpha \circ K\widetilde{+}_{p}\beta \circ L)^{sp} \leq \alpha\mu(K)^{sp} + \beta \mu(L)^{sp}.
\end{align}
If $p<0$ or $p> n+\frac{1}{r}$, then inequality (\ref{bmi}) is reversed. Equality holds if and only if $K$ and $L$ are dilatates.
\end{lemma}

\section{General $L_p$ $\mu$-projection bodies and General $L_p$ $\mu$-centroid bodies}\label{SEC3}

In this part, we will introduce the concepts of nonsymmetric $L_p$ $\mu$-projection body and nonsymmetric $L_p$ $\mu$-centroid body, and then study the corresponding measure comparison problem according to their properties.

The nonsymmetric $L_p$ cosine transform of $\nu$ is defined by Haberl and Schuster \cite{HSFE} as follows: For each finite Borel measure $\nu$ on $\mathbb{S}^{n-1}$, a continuous function $\mathcal{C}_p^+\nu$ on $\mathbb{S}^{n-1}$, the nonsymmetric $L_p$ cosine transform of $\nu$, is defined by
\begin{align*}(\mathcal{C}_p^+\nu)(u) = c_{n,p}\int_{\mathbb{S}^{n-1}} (u \cdot v)_+^p d\nu(v), \qquad u \in \mathbb{S}^{n-1}, \end{align*}
where $(u\cdot v)_+=\max\{u\cdot v, 0\}$, and
\begin{align*}
c_{n,p}=\frac{\Gamma\left(\mbox{$\frac{n+p}{2}$}\right)}{\pi^\frac{n-1}{2}\Gamma\left(\frac{1+p}{2}\right)}.
\end{align*}
According to the nonsymmetric $L_p$ cosine transform, we give the following concept of nonsymmetric $L_p$ $\mu$-projection body by: Let $\mu$ satisfy the homogeneity condition {\bf ({\rm I})}, $K \in \mathscr{K}^n_{\mathrm{o}}$ and  $p \geq 1$, then the nonsymmetric $L_p$ $\mu$-projection body $\Pi_{\mu,p}^+K$ of $K$, is the convex body defined by
\begin{align*}
h(\Pi_{\mu,p}^+K,\cdot)^p=\mathcal{C}_p^+S_{\mu,p}(K,\cdot),
\end{align*}
i.e., for $u \in \mathbb{S}^{n-1}$,
\begin{align*}
h(\Pi_{\mu,p}^+K,u)^p=c_{n,p}\int_{\mathbb{S}^{n-1}} (u \cdot v)_+^p dS_{\mu,p}(K,v).
\end{align*}
 Let $(u \cdot v)_{-}=((-u) \cdot v)_{+}=\max \{(-u) \cdot v, 0\}$. For $K \in
\mathscr{K}^n_{\mathrm{o}}$ and  $p \geq 1$, we define $\Pi_{\mu,p}^-K$ by 
\begin{align*}
h(\Pi_{\mu,p}^-K,u)^p=c_{n,p}\int_{\mathbb{S}^{n-1}} (u \cdot v)_-^p dS_{\mu,p}(K,v),
\end{align*}
for $u \in \mathbb{S}^{n-1}$. Thus, we have
\begin{align*}
h(-\Pi_{\mu,p}^+K,u)^p&=h(\Pi_{\mu,p}^+K,-u)^p\\
&=c_{n,p}\int_{\mathbb{S}^{n-1}} (-u \cdot v)_+^p dS_{\mu,p}(K,v)\\
&=c_{n,p}\int_{\mathbb{S}^{n-1}} (u \cdot v)_-^p dS_{\mu,p}(K,v)\\
&=h(\Pi_{\mu,p}^-K,u)^p,
\end{align*}
that is,
\begin{align}\label{tzfgx}
\Pi_{\mu,p}^-K=-\Pi_{\mu,p}^+K. 
\end{align}

Let $\mu$ satisfy the homogeneity condition {\bf ({\rm I})}, $K \in \mathscr{S}^n_o$ and  $p \geq 1$, then the nonsymmetric $L_p$ $\mu$-centroid body $\Gamma_{\mu,p}^+K$ of $K$ is a convex body defined by
\begin{align}\label{zfz}
h(\Gamma_{\mu,p}^+K,u)^p=\frac{2}{c_{n,p}\mu(K)}\int_K (u \cdot x )_+^p d\mu(x),  \qquad u \in \mathbb{S}^{n-1}.
\end{align}
Using the polar coordinates, it is easy to verify that for all $u \in \mathbb{S}^{n-1}$,
\begin{align}\label{jzfz}
h(\Gamma_{\mu,p}^+K,u)^p=\frac{2}{(p+\frac{1}{s})c_{n,p}\mu(K)}\int_{\mathbb{S}^{n-1}} (u \cdot v )_+^p\rho(K,v)^{p+\frac{1}{s}}d \mu(v).
\end{align}
Similarly, we also define $\Gamma_{\mu,p}^-K$ by
\begin{align}\label{ffz}
h(\Gamma_{\mu,p}^-K,u)^p=\frac{2}{c_{n,p}\mu(K)}\int_K (u \cdot x )_-^p d\mu(x),  \qquad u \in \mathbb{S}^{n-1}.
\end{align}
Obviously,  for $u \in \mathbb{S}^{n-1}$, from (\ref{zfz}) and (\ref{ffz}),we easily obtain
\begin{align*}
h(-\Gamma_{\mu,p}^+K,u)^p&=h(\Gamma_{\mu,p}^+K,-u)^p=\frac{2}{c_{n,p}\mu(K)}\int_K ((-u) \cdot x )_+^p d\mu(x) \\
&=\frac{2}{c_{n,p}\mu(K)}\int_K (u \cdot x )_-^p d\mu(x)\\
&=h(\Gamma_{\mu,p}^-K,u)^p,
\end{align*}
i.e.,
\begin{align}\label{zxtds}
\Gamma_{\mu,p}^-K =- \Gamma_{\mu,p}^+K.
\end{align}

In 2006, Ludwig \cite{LUDA} introduced a function $\psi_{\tau}:
\mathbb{R} \rightarrow [0,\infty)$ by
\begin{align*}\psi_{\tau}(t)=|t|+\tau t,\qquad \tau \in [-1,1].
\end{align*}
Let $\mu$ satisfy the homogeneity condition {\bf ({\rm I})}, $K \in \mathscr{K}_{\mathrm{o}}^n$,  $p \geq 1$ and $\tau \in [-1,1]$, the general $L_p$ $\mu$-projection body $\Pi_{\mu,p}^{\tau}K
\in \mathscr{K}_{\mathrm{o}}^n$ of $K$ is defined by  
\begin{align} \label{ghmt}
h(\Pi_{\mu,p}^{\tau}K,u)^p=c_{n,p}(\tau)\int_{\mathbb{S}^{n-1}}\psi_{\tau}(u\cdot
v)^p dS_{\mu,p}(K,v), \qquad u \in \mathbb{S}^{n-1},
\end{align}
where
\begin{align*}
c_{n,p}(\tau)=\frac{c_{n,p}}{(1+\tau)^p+(1-\tau)^p}. 
\end{align*}
Usually, we use $\Pi_{\mu,p}^{\tau,\ast}K$ to represent the polar body of $\Pi_{\mu,p}^{\tau}K$ instead of $(\Pi_{\mu,p}^{\tau}K)^\ast$. 

Since $\mu$ satisfies the homogeneity condition {\bf ({\rm I})}, for $\lambda>0$, from (\ref{lphmt2}) and (\ref{ghmt}), it follows that
\begin{align*}
\Pi_{\mu,p}^{\tau}(\lambda K)=\lambda^\frac{1-sp}{sp}\Pi_{\mu,p}^{\tau}K. 
\end{align*}

For all $u,v\in \mathbb{S}^{n-1}$, it is easy to verify that
\begin{align}\label{csfc}
\psi_{\tau}(u\cdot
v)^p=(1+\tau)^p(u \cdot v )_+^p +(1-\tau)^p(u \cdot v )_-^p . 
\end{align}

From (\ref{csfc}) and the definition of $\Pi_{\mu,p}^{\pm}K$, we obtain
\begin{align}\label{tytmz}
\Pi_{\mu,p}^{\tau}K=g_1(\tau)\cdot\Pi_{\mu,p}^+K +_p g_2(\tau)\cdot\Pi_{\mu,p}^-K.
\end{align}
where 
\begin{align}\label{g1} 
g_1(\tau)=\frac{(1+\tau)^p}{(1+\tau)^p+(1-\tau)^p}, \qquad g_2(\tau)=\frac{(1-\tau)^p}{(1+\tau)^p+(1-\tau)^p}.
\end{align}
Obviously, 
\begin{align}\label{g2}
g_1(\tau)+g_2(\tau)=1, \qquad g_1(-\tau)=g_2(\tau), \qquad g_1(\tau)=g_2(-\tau).
\end{align}
Note that, let $\tau=\pm1$, then 
\begin{align*}
\Pi_{\mu,p}^{+1}K=\Pi_{\mu,p}^{+}K, \qquad \Pi_{\mu,p}^{-1}K=\Pi_{\mu,p}^{-}K.
\end{align*}
If $\tau=0$, we define that $\Pi_{\mu,p}^{0}K=\Pi_{\mu,p}K$, by (\ref{tytmz}) and (\ref{g1}), then
\begin{align}\label{tytzf}
\Pi_{\mu,p}K=\frac{1}{2}\cdot\Pi_{\mu,p}^{+}K+_p \frac{1}{2}\cdot \Pi_{\mu,p}^{-}K.
\end{align}

Let $\mu$ satisfy the homogeneity condition {\bf ({\rm I})}, $K \in \mathscr{S}^n_o$,  $p \geq 1$ and $\tau \in [-1,1]$, the general $L_p$ $\mu$-centroid body $\Gamma_{\mu,p}^{\tau}K\in \mathscr{K}^n_{\mathrm{o}}$ of $K$ is defined by
\begin{align}\label{ghmz}
h(\Gamma_{\mu,p}^{\tau}K,u)^p=&\frac{2}{\alpha_{n,p}(\tau)\mu(K)}\int_K \psi_{\tau}(u\cdot
x)^pd\mu(x)\\
\nonumber=&\frac{2}{\alpha_{n,p}(\tau)(p+\frac{1}{s})\mu(K)}\int_{\mathbb{S}^{n-1}} \psi_{\tau}(u\cdot
v)^p\rho(K,v)^{p+\frac{1}{s}}d\mu(v),
\end{align}
for $u \in \mathbb{S}^{n-1}$, and
\begin{align*}
\alpha_{n,p}(\tau)=c_{n,p}[(1+\tau)^p+(1-\tau)^p]. 
\end{align*}
Then for every $\tau \in [-1,1]$, by (\ref{ghmz}), (\ref{csfc}) and (\ref{g1}), we can obtain
\begin{align} \label{mtaupm}
\Gamma_{\mu,p}^{\tau}K=g_1(\tau)\cdot
\Gamma_{\mu,p}^+K +_p g_2(\tau)\cdot
\Gamma_{\mu,p}^-K.
\end{align}
Usually, we use $\Gamma_{\mu,p}^{\tau,\ast}K$ to represent the polar body of $\Gamma_{\mu,p}^{\tau}K$ other than $(\Gamma_{\mu,p}^{\tau}K)^\ast$.

Since $\mu$ satisfies the homogeneity condition {\bf ({\rm I})}, for $\lambda>0$, from (\ref{ghmz}), it follows that
\begin{align*}
\Gamma_{\mu,p}^{\tau}(\lambda K)=\lambda\Gamma_{\mu,p}^{\tau}K. 
\end{align*}

Note that, let $\tau=\pm1$, then 
\begin{align*}
\Gamma_{\mu,p}^{+1}K=\Gamma_{\mu,p}^{+}K, \qquad \Gamma_{\mu,p}^{-1}K=\Gamma_{\mu,p}^{-}K.
\end{align*}
If $\tau=0$, we define that $\Gamma_{\mu,p}^{0}K=\Gamma_{\mu,p}K$, then 
\begin{align}\label{tytzfc}
\Gamma_{\mu,p}K=\frac{1}{2}\cdot\Gamma_{\mu,p}^{+}K+_p \frac{1}{2}\cdot \Gamma_{\mu,p}^{-}K.
\end{align}

Next, the measure comparison theorems of the general $L_p$ $\mu$-projection body and the general $L_p$ $\mu$-centroid body are obtained. Before that, we give the following properties.

\begin{proposition}\label{prop1}
Let $\mu$ satisfy the homogeneity condition {\bf ({\rm I})}, if $K,L \in \mathscr{K}^n_{\mathrm{o}}$, $ p\geq 1$ and $\tau \in [-1,1]$, then 
\begin{align*}
V_{\mu,p}(K,\Pi_{\mu,p}^{\tau}L)=V_{\mu,p}(L,\Pi_{\mu,p}^{\tau}K).
\end{align*}
\end{proposition} 
\begin{proof}
From (\ref{lphmt}) together with the (\ref{ghmt}) and Fubini's theorem, we obtain 
\begin{align*}
V_{\mu,p}(K,\Pi_{\mu,p}^{\tau}L)&=s\int_{\mathbb{S}^{n-1}} h(\Pi_{\mu,p}^{\tau}L,u)^p d S_{\mu, p}(K, u) \\
&=sc_{n,p}(\tau)\int_{\mathbb{S}^{n-1}}\int_{\mathbb{S}^{n-1}}\psi_{\tau}(u\cdot
v)^p dS_{\mu,p}(L,v) d S_{\mu, p}(K, u),\\
&=sc_{n,p}(\tau)\int_{\mathbb{S}^{n-1}}\int_{\mathbb{S}^{n-1}}\psi_{\tau}(v\cdot
u)^p d S_{\mu, p}(K, u)dS_{\mu,p}(L,v) ,\\
&=s\int_{\mathbb{S}^{n-1}} h(\Pi_{\mu,p}^{\tau}K,v)^p d S_{\mu, p}(L, v) \\
&=V_{\mu,p}(L,\Pi_{\mu,p}^{\tau}K).
\end{align*}

\end{proof}

\begin{proposition} \label{durch} 
Let $\mu$ satisfy the homogeneity condition {\bf ({\rm I})}, if $K \in \mathscr{K}^n_{\mathrm{o}}$, $L \in \mathscr{S}^n_o$, $p\geq 1$ and $\tau \in [-1,1]$, then
\begin{align*}
V_{\mu,p}(K,\Gamma_{\mu,p}^{\tau}L)=\frac{2}{c_{n,p}(\tau)\alpha_{n,p}(\tau)(p+\frac{1}{s})\mu(L)}\widetilde{V}_{\mu,-p}(L, \Pi_{\mu,p}^{\tau,*}K). 
\end{align*}
\end{proposition}
\begin{proof} 
Since $K \in \mathscr{K}^n_{\mathrm{o}}$ and $L \in
\mathscr{S}^n_o$, by (\ref{lphmt}) and (\ref{ghmz}), we obtain
\begin{align*}
V_{\mu,p}(K,\Gamma_{\mu,p}^{\tau}L)&=s\int_{\mathbb{S}^{n-1}} h(\Gamma_{\mu,p}^{\tau}L,u)^p d S_{\mu, p}(K, u) \\
&=\frac{2s}{\alpha_{n,p}(\tau)(p+\frac{1}{s})\mu(L)}\int_{\mathbb{S}^{n-1}}\int_{\mathbb{S}^{n-1}} \psi_{\tau}(u\cdot
v)^p\rho(L,v)^{p+\frac{1}{s}}d\mu(v) d S_{\mu, p}(K, u),\end{align*}
from Fubini's theorem, (\ref{ghmt}), (\ref{SRF}) and (\ref{LpHT}), it follow that
\begin{align*}
V_{\mu,p}(K,\Gamma_{\mu,p}^{\tau}L)&=\frac{2s}{\alpha_{n,p}(\tau)(p+\frac{1}{s})\mu(L)}\int_{\mathbb{S}^{n-1}}\int_{\mathbb{S}^{n-1}} \psi_{\tau}(v\cdot
u)^pd S_{\mu, p}(K, u)\rho(L,v)^{p+\frac{1}{s}}d\mu(v) \\
&=\frac{2s}{c_{n,p}(\tau)\alpha_{n,p}(\tau)(p+\frac{1}{s})\mu(L)}\int_{\mathbb{S}^{n-1}}\rho(L,v)^{p+\frac{1}{s}}\rho(\Pi_{\mu,p}^{\tau,*}K, v)^{-p}d\mu(v) \\
&=\frac{2}{c_{n,p}(\tau)\alpha_{n,p}(\tau)(p+\frac{1}{s})\mu(L)}\widetilde{V}_{\mu,-p}(L, \Pi_{\mu,p}^{\tau,*}K).
\end{align*}
This leads to the desired result.
\end{proof}

Next, let $\mathcal{P}^{\tau}_{\mu,p}$ and $\mathcal{P}^{\tau,*}_{\mu,p}$ denotes the set of general $L_p$ $\mu$-projection bodies and their polar bodies, respectively. Let $\mathcal{C}^{\tau}_{\mu,p}$ and $\mathcal{C}^{\tau,*}_{\mu,p}$ denotes the set of general $L_p$ $\mu$-centroid bodies and their polar bodies, respectively. 

\begin{theorem}\label{tyttjb}
 Suppose that $\mu$ satisfies the homogeneity condition {\bf ({\rm I})} and concavity condition {\bf ({\rm II})}, let $K \in \mathscr{K}^n_{\mathrm{o}}$, $ p\geq 1$ and $\tau \in [-1,1]$. 

(a) If $\Pi_{\mu,p}^{\tau}K\subseteq \Pi_{\mu,p}^{\tau}L$ and $L\in\mathcal{P}^{\tau}_{\mu,p}$, then
\begin{align*}
1\geq p \left(\frac{\mu(K)}{\mu(L)}\right)^{1-s}+(1-p)\frac{\mu(K)}{\mu(L)}.
\end{align*}

(b) If $\Pi_{\mu,p}^{\tau,\ast}K\supseteq \Pi_{\mu,p}^{\tau,\ast}L$ and $L\in\mathcal{C}^{\tau}_{\mu,p}$, then
\begin{align*}
1\geq p \left(\frac{\mu(K)}{\mu(L)}\right)^{1-s}+(1-p)\frac{\mu(K)}{\mu(L)}.
\end{align*}
\end{theorem}

\begin{proof}
For (a), since $\Pi_{\mu,p}^{\tau}K\subseteq \Pi_{\mu,p}^{\tau}L$ and $L\in\mathcal{P}^{\tau}_{\mu,p}$, then there exist $M \in\mathscr{K}^n_o$, such that $L=\Pi_{\mu,p}^{\tau}M$, by (\ref{lphmt}), we deduce
\begin{align*}
 V_{\mu,p}(M, \Pi_{\mu,p}^{\tau}K)\leq V_{\mu,p}(M, \Pi_{\mu,p}^{\tau}L).
\end{align*}
From these and Proposition \ref{prop1}, it follows that
\begin{align*}
 V_{\mu,p}(K, \Pi_{\mu,p}^{\tau}M)= V_{\mu,p}(M, \Pi_{\mu,p}^{\tau}K)\leq V_{\mu,p}(M, \Pi_{\mu,p}^{\tau}L)= V_{\mu,p}(L, \Pi_{\mu,p}^{\tau}M).
\end{align*}
Let $\Pi_{\mu,p}^{\tau}M=L$ in the  above inequality, from (\ref{lphtm2}), it follows that
\begin{align*}
\mu(L)\geq V_{\mu,p}(K, L)\geq p\mu(K)^{1-s}\mu(L)^{s}+(1-p)\mu(K),
\end{align*}
i.e.,
\begin{align*}
1\geq p \left(\frac{\mu(K)}{\mu(L)}\right)^{1-s}+(1-p)\frac{\mu(K)}{\mu(L)}.
\end{align*}
Therefore, (a) is established.

For (b), since $\Pi_{\mu,p}^{\tau,\ast}K\supseteq \Pi_{\mu,p}^{\tau,\ast}L$ and $L\in\mathcal{C}^{\tau}_{\mu,p}$, then there exist $N \in\mathscr{S}^n_o$, such that $L=\Gamma_{\mu,p}^{\tau}N$, from (\ref{LpHT}), we obtain
\begin{align*}
 \widetilde{V}_{\mu,-p}(N, \Pi_{\mu,p}^{\tau,\ast}K)\leq \widetilde{V}_{\mu,-p}(N, \Pi_{\mu,p}^{\tau,\ast}L),
\end{align*}
from this and Proposition \ref{durch}, we deduce
\begin{align*}
V_{\mu,p}(K, \Gamma_{\mu,p}^{\tau}N)\leq V_{\mu,p}(L, \Gamma_{\mu,p}^{\tau}N).
\end{align*}
Let $\Gamma_{\mu,p}^{\tau}N=L$ in the above inequality, from (\ref{lphtm2}), it follows that
\begin{align*}
\mu(L)\geq V_{\mu,p}(K, L)\geq p\mu(K)^{1-s}\mu(L)^{s}+(1-p)\mu(K),
\end{align*}
i.e.,
\begin{align*}
1\geq p \left(\frac{\mu(K)}{\mu(L)}\right)^{1-s}+(1-p)\frac{\mu(K)}{\mu(L)}.
\end{align*}
This completes the proof of the Theorem.
\end{proof}

Especially, let $p=1$ and $\tau=0$ in Theorem \ref{tyttjb}, since $1-s>0$, we can get the following interesting results.

\begin{corollary}\label{CORR1}
Let $\mu$ satisfy the homogeneity condition {\bf ({\rm I})} and concavity condition {\bf ({\rm II})} and $K \in \mathscr{K}^n_{\mathrm{o}}$.
 
(a) If $\Pi_{\mu}K\subseteq \Pi_{\mu}L$ and $L\in\mathcal{P}_{\mu}$, then
\begin{align*}
\mu(K)\leq\mu(L).
\end{align*}

(b) If $\Pi_{\mu}^{\ast}K\supseteq \Pi_{\mu}^{\ast}L$ and $L\in\mathcal{C}_{\mu}$, then
\begin{align*}
\mu(K)\leq\mu(L).
\end{align*}
Here, $\mathcal{P}_{\mu}$ is the set of $\mu$-projection bodies and $\mathcal{C}_{\mu}$ denotes the set of $\mu$-centroid bodies.
\end{corollary}

\begin{proposition}\label{prop3}
Let $\mu$ satisfy the homogeneity condition {\bf ({\rm I})}, $K,L \in \mathscr{S}^n_{\mathrm{o}}$ and  $p \geq 1$, then
\begin{align*}
\frac{ \widetilde{V}_{\mu,-p}(K,\Gamma_{\mu,p}^{\tau,\ast} L)}{\mu(K)}=\frac{ \widetilde{V}_{\mu,-p}(L,\Gamma_{\mu,p}^{\tau,\ast} K)}{\mu(L)}.
\end{align*}
\end{proposition}
\begin{proof}

Combining (\ref{LpHT}), (\ref{ghmz}), (\ref{SRF}) and Fubini's theorem, using (\ref{SRF}) (\ref{ghmz}) and (\ref{LpHT}) again, we get
\begin{align*}
&\widetilde{V}_{\mu,-p}(L,\Gamma_{\mu,p}^{\tau,\ast} K)\\
=&s\int_{\mathbb{S}^{n-1}}\rho(L,u)^\frac{1+sp}{s}\rho(\Gamma_{\mu,p}^{\tau,\ast} K,u)^{-p}d\mu(u)\\
=&\frac{2s}{\alpha_{n,p}(\tau)(p+\frac{1}{s})\mu(K)}\int_{\mathbb{S}^{n-1}}\int_{\mathbb{S}^{n-1}} \psi_{\tau}(u\cdot
v)^p\rho(K,v)^\frac{1+sp}{s}\rho(L,u)^\frac{1+sp}{s}d\mu(v)d\mu(u)\\=&\frac{2s}{\alpha_{n,p}(\tau)(p+\frac{1}{s})\mu(K)}\int_{\mathbb{S}^{n-1}}\int_{\mathbb{S}^{n-1}} \psi_{\tau}(v\cdot
u)^p\rho(L,u)^\frac{1+sp}{s}\rho(K,v)^\frac{1+sp}{s}d\mu(u)d\mu(v)\\
=&\frac{\mu(L)}{\mu(K)}s\int_{\mathbb{S}^{n-1}}\rho(K,v)^\frac{1+sp}{s}\rho(\Gamma_{\mu,p}^{\tau,\ast} L,v)^{-p}d\mu(v)\\
=&\frac{\mu(L)}{\mu(K)} \widetilde{V}_{\mu,-p}(K,\Gamma_{\mu,p}^{\tau,\ast} L). \end{align*}
This gives the equality.
\end{proof}

\begin{theorem}\label{ty2}
 Let $\mu$ satisfy the homogeneity condition {\bf ({\rm I})}, $K\in \mathscr{S}^n_{\mathrm{o}}$, $ p\geq 1$ and $\tau \in [-1,1]$.

(a) If $\Gamma_{\mu,p}^{\tau}K\subseteq \Gamma_{\mu,p}^{\tau}L$ and $L\in\mathcal{P}^{\tau,*}_{\mu,p}$, then
\begin{align*}
\mu(K)\leq\mu(L).
\end{align*}

(a) If $\Gamma_{\mu,p}^{\tau,\ast}K\supseteq \Gamma_{\mu,p}^{\tau,\ast}L$ and $L\in\mathcal{C}^{\tau,*}_{\mu,p}$, then
\begin{align*}
\mu(K)\leq\mu(L).
\end{align*}
Equality holds in (a) and (b) if and only if $K=L$.
\end{theorem}

\begin{proof}
For (a), since $\Gamma_{\mu,p}^{\tau}K\subseteq \Gamma_{\mu,p}^{\tau}L$ and $L\in\mathcal{P}^{\tau,*}_{\mu,p}$, then there exist $Q \in\mathscr{K}^n_o$, such that $L=\Pi_{\mu,p}^{\tau,*}Q$, by (\ref{lphmt}), we get
\begin{align}\label{VQL1}
 V_{\mu,p}(Q, \Gamma_{\mu,p}^{\tau}K)\leq V_{\mu,p}(Q, \Gamma_{\mu,p}^{\tau}L)
\end{align}
with equality if and only if $ \Gamma_{\mu,p}^{\tau}K= \Gamma_{\mu,p}^{\tau}L$. From these and Proposition \ref{prop3}, it follows that
\begin{align*}
\frac{1}{\mu(K)}\widetilde{V}_{\mu,-p}(K, \Pi_{\mu,p}^{\tau,*}Q)\leq\frac{1}{\mu(L)}\widetilde{V}_{\mu,-p}(L, \Pi_{\mu,p}^{\tau,*}Q).
\end{align*}
Let $\Pi_{\mu,p}^{\tau,*}Q=L$ in the above inequality, from Lemma \ref{DMMBS}, it follows that
\begin{align*}
1\geq\frac{1}{\mu(K)}\widetilde{V}_{\mu,-p}(K, L)\geq\frac{1}{\mu(K)} \mu(K)^{1+sp}\mu(L)^{-sp}.
\end{align*}
i.e.,
\begin{align*}
\mu(K)\leq\mu(L).
\end{align*}
The equality conditions of Lemma \ref{DMMBS} and (\ref{VQL1}) imply that in (a) the equality holds if and only if $K=L$.
This shows that (a) holds.

For (b), if $\Gamma_{\mu,p}^{\tau,\ast}K\supseteq \Gamma_{\mu,p}^{\tau,\ast}L$ and $L\in\mathcal{C}^{\tau,*}_{\mu,p}$, then there exist $M \in\mathscr{S}^n_o$, such that $L=\Gamma_{\mu,p}^{\tau,*}M$, by (\ref{LpHT}) and (\ref{RC}), we deduce
\begin{align}\label{dlpqbi}
 \widetilde{V}_{\mu,-p}(M, \Gamma_{\mu,p}^{\tau,\ast}K)\leq \widetilde{V}_{\mu,-p}(M, \Gamma_{\mu,p}^{\tau,\ast}L),
\end{align}
with equality if and only if $ \Gamma_{\mu,p}^{\tau,\ast}K= \Gamma_{\mu,p}^{\tau,\ast}L$.
From these and Proposition \ref{durch}, it follows that
\begin{align*}
\frac{1}{\mu(K)}\widetilde{V}_{\mu,-p}(K, \Gamma_{\mu,p}^{\tau,*}M)\leq\frac{1}{\mu(L)}\widetilde{V}_{\mu,-p}(L, \Gamma_{\mu,p}^{\tau,*}M). 
\end{align*}
Let $\Gamma_{\mu,p}^{\tau,*}M=L$ in the above inequality, from Lemma \ref{DMMBS}, it follows that
\begin{align*}
1\geq\frac{1}{\mu(K)}\widetilde{V}_{\mu,-p}(K, L)\geq\frac{1}{\mu(K)} \mu(K)^{1+sp}\mu(L)^{-sp},
\end{align*}
equivalently,
\begin{align*}
\mu(K)\leq\mu(L).
\end{align*}
By the equality conditions of Lemma \ref{DMMBS} and (\ref{dlpqbi}), we see that equality holds in (b) if and only if $K=L$.
\end{proof}

Especially, let $p=1$ and $\tau=0$ in Theorem \ref{ty2}, we can get the following corollary.

\begin{corollary} 
Let $\mu$ satisfy the homogeneity condition {\bf ({\rm I})} and $K\in \mathscr{S}^n_{\mathrm{o}}$. 
 
(a) If $\Gamma_{\mu}K\subseteq \Gamma_{\mu}L$ and $L\in\mathcal{P}^{*}_{\mu}$, then
\begin{align*}
\mu(K)\leq\mu(L).
\end{align*}

(b) If $\Gamma_{\mu}^{\ast}K\supseteq \Gamma_{\mu}^{\ast}L$ and $L\in\mathcal{C}^{*}_{\mu}$, then
\begin{align*}
\mu(K)\leq\mu(L).
\end{align*}
Equality holds in (a) and (b) if and only if $K=L$. Here, $\mathcal{P}^{*}_{\mu}$ and $\mathcal{C}^{*}_{\mu}$ denotes the set of polar bodies of $\mu$-projection bodies and $\mu$-centroid bodies, respectively. 
\end{corollary}

\section{Extremal Inequalities for general $L_p$ $\mu$-projection body and $\mu$-centroid body}\label{SEC4}
In this section, we will give the proof of Theorem \ref{TLD1} and Theorem \ref{TLD2}, which are Theorem \ref{dliqw} and Theorem \ref{dliqw2}. In order to prove the theorems, the following properties are required.

\begin{proposition}
Let $\mu$ satisfy the homogeneity condition {\bf ({\rm I})}, if $K\in \mathscr{K}^n_o$, $L\in \mathscr{S}^n_o$, $p\geq 1$ and $\tau \in [-1,1]$, then

\begin{align}\label{pro1}
\Pi_{\mu, p}^{-\tau} K=-\Pi_{\mu, p}^{\tau} K,
\end{align}
and
\begin{align}\label{pro2}
\Gamma_{\mu,p}^{-\tau}L=-\Gamma_{\mu,p}^{\tau}L.
\end{align}

\end{proposition}
\begin{proof}
From (\ref{tytmz}), (\ref{g2}) and (\ref{tzfgx}), we have that
\begin{align*}
\Pi_{\mu, p}^{-\tau} K &=g_1(-\tau)\cdot \Pi_{\mu, p}^{+} K+_{p} g_2(-\tau)\cdot \Pi_{\mu, p}^{-} K \\
&=g_2(\tau)\cdot (-\Pi_{\mu, p}^{-}K)+_{p} g_1(\tau)\cdot(- \Pi_{\mu, p}^{+}K),
\end{align*}
from this, (\ref{lpmz}) and (\ref{tytmz}), for $ u \in \mathbb{S}^{n-1}$, we have
\begin{align*}
h(\Pi_{\mu, p}^{-\tau} K ,u)^p&=g_2(\tau)h(-\Pi_{\mu, p}^{-}K,u)^p+g_1(\tau)h(-\Pi_{\mu, p}^{+}K,u)^p\\
&=g_2(\tau)h(\Pi_{\mu, p}^{-}K,-u)^p+g_1(\tau)h(\Pi_{\mu, p}^{+}K,-u)^p\\
&=h(g_1(\tau)\cdot \Pi_{\mu, p}^+K+_{p} g_2(\tau)\cdot \Pi_{\mu, p}^-K,-u)^p\\
&=h(-\Pi_{\mu, p}^{\tau} K ,u)^p.
\end{align*}
This obtain
\begin{align*}
\Pi_{\mu, p}^{-\tau}K=-\Pi_{\mu, p}^{\tau} K.
\end{align*}
Similar to the above method, from (\ref{mtaupm}), (\ref{g2}), (\ref{zxtds}), (\ref{lpmz}) and using (\ref{mtaupm}) again, we can easily verify that (\ref{pro2}) holds.
\end{proof}

\begin{proposition}
Let $\mu$ satisfy the homogeneity condition {\bf ({\rm I})}, if $K\in \mathscr{K}^n_o$, $L\in \mathscr{S}^n_o$, $p\geq 1$, $\tau \in [-1,1]$ and $\tau \neq 0$, then
\begin{align*}
\Pi_{\mu, p}^{\tau} K=\Pi_{\mu, p}^{-\tau} K \quad \Longleftrightarrow \quad \Pi_{\mu, p}^{+} K=\Pi_{\mu, p}^{-} K;
\end{align*}
and 
\begin{align*}
\Gamma_{\mu, p}^{\tau} L=\Gamma_{\mu, p}^{-\tau} L \quad \Longleftrightarrow \quad \Gamma_{\mu, p}^{+} L=\Gamma_{\mu, p}^{-} L .
\end{align*}
\end{proposition}

\begin{proof}
Since $K,L \in \mathscr{K}^n_o$, $p\geq 1$, $\tau \in [-1,1]$ and $\tau \neq 0$, then by (\ref{tytmz}) and (\ref{g2}), we deduce
\begin{align*}
\Pi_{\mu, p}^{-\tau} K=g_2(\tau)\cdot \Pi_{\mu, p}^{+} K+_{p} g_1(\tau)\cdot \Pi_{\mu, p}^{-} K,
\end{align*}
from this and (\ref{lpmz}), for all $u \in \mathbb{S}^{n-1}$, we obtain
\begin{align}\label{ftyt}
h(\Pi_{\mu,p}^{-\tau} K,u)^p=g_2(\tau) h(\Pi_{\mu,p}^+K,u)^p+g_1(\tau) h(\Pi_{\mu,p}^-K,u)^p.
\end{align}
When $\Pi_{\mu, p}^{+} K=\Pi_{\mu, p}^{-} K$, by (\ref{ftyt}), (\ref{tytmz}) and (\ref{lpmz}), 
for all $u \in \mathbb{S}^{n-1}$, then
\begin{align*}
h(\Pi_{\mu,p}^{-\tau} K,u)^p=h(\Pi_{\mu,p}^{\tau} K,u)^p,
\end{align*}
thus $\Pi_{\mu, p}^{\tau} K=\Pi_{\mu, p}^{-\tau} K$.

Conversely, if $\Pi_{\mu, p}^{\tau} K=\Pi_{\mu, p}^{-\tau} K$, from (\ref{tytmz}) and (\ref{lpmz}), for all $u \in \mathbb{S}^{n-1}$, we have
\begin{align}\label{klmh}
h(\Pi_{\mu,p}^{\tau} K,u)^p=g_1(\tau) h(\Pi_{\mu,p}^+K,u)^p+g_2(\tau) h(\Pi_{\mu,p}^-K,u)^p.
\end{align}
From (\ref{klmh}) and (\ref{ftyt}), 
for all $u \in \mathbb{S}^{n-1}$, we obtain 
\begin{align*}
(g_{1}(\tau)-g_{2}(\tau)) h(\Pi_{\mu,p}^+K,u)^p=(g_{1}(\tau)-g_{2}(\tau)) h(\Pi_{\mu,p}^-K,u)^p,
\end{align*}
Since $g_{1}(\tau)-g_{2}(\tau) \neq 0$ when $\tau \neq 0$, we get $\Pi_{\mu, p}^{+} K=\Pi_{\mu, p}^{-} K$.

In addition, from (\ref{mtaupm}), (\ref{g2}) and (\ref{lpmz}), we obtain 
\begin{align}\label{ftyt3}
h(\Gamma_{\mu,p}^{\tau} L,u)^p=g_1(\tau) h(\Gamma_{\mu,p}^+L,u)^p+g_2(\tau) h(\Gamma_{\mu,p}^-L,u)^p,
\end{align}
and
\begin{align}\label{ftyt4}
h(\Gamma_{\mu,p}^{-\tau} L,u)^p=g_2(\tau) h(\Gamma_{\mu,p}^+L,u)^p+g_1(\tau) h(\Gamma_{\mu,p}^-L,u)^p.
\end{align}

If $\Gamma_{\mu, p}^{+} L=\Gamma_{\mu, p}^{-} L$, by (\ref{ftyt4}) and (\ref{mtaupm}), we can obtain $h(\Gamma_{\mu, p}^{\tau}L,u)^p =h(\Gamma_{\mu, p}^{-\tau} L,u)^p$ 
for all $u \in \mathbb{S}^{n-1}$, i.e., $\Gamma_{\mu, p}^{\tau}L =\Gamma_{\mu, p}^{-\tau}L$.

On the other hand, If $\Gamma_{\mu, p}^{\tau}L =\Gamma_{\mu, p}^{-\tau}L$, by (\ref{ftyt3}) and (\ref{ftyt4}), we obtain 
\begin{align*}
(g_{1}(\tau)-g_{2}(\tau)) h(\Gamma_{\mu,p}^+L,u)^p=(g_{1}(\tau)-g_{2}(\tau)) h(\Gamma_{\mu,p}^-L,u)^p,
\end{align*}
where $u \in \mathbb{S}^{n-1}$. Since $\tau \neq 0$, then $g_{1}(\tau)-g_{2}(\tau) \neq 0$. Hence, $\Gamma_{\mu, p}^{+} L=\Gamma_{\mu, p}^{-} L$.
\end{proof}

\begin{proposition}
Let $\mu$ satisfy the homogeneity condition {\bf ({\rm I})}, if $K\in \mathscr{K}^n_o$, $L\in \mathscr{S}^n_o$, $p\geq 1$ and $\tau \in [-1,1]$, then
\begin{align}\label{zftzfa}
\Pi_{\mu, p}^{\tau} K+_p\Pi_{\mu, p}^{-\tau} K = \Pi_{\mu, p}^{+} K+_p\Pi_{\mu, p}^{-} K,
\end{align}
and
\begin{align}\label{zftzfb}
\Gamma_{\mu, p}^{\tau} L+_p\Gamma_{\mu, p}^{-\tau} L = \Gamma_{\mu, p}^{+} L+_p\Gamma_{\mu, p}^{-} L.
\end{align}
\end{proposition}

\begin{proof}
By (\ref{klmh}), (\ref{ftyt}) and (\ref{g2}), for all $u \in \mathbb{S}^{n-1}$, we have
\begin{align*}
h(\Pi_{\mu, p}^{\tau} K,u)^p+ h(\Pi_{\mu, p}^{-\tau} K,u)^p = h(\Pi_{\mu, p}^{+} K,u)^p+ h(\Pi_{\mu, p}^{-} K,u)^p,
\end{align*}
from this and (\ref{lpmz}), we obtain 
\begin{align*}
h(\Pi_{\mu, p}^{\tau} K+_p\Pi_{\mu, p}^{-\tau} K,u)^p = h(\Pi_{\mu, p}^{+} K+_p \Pi_{\mu, p}^{-} K,u)^p,
\end{align*}
that is
\begin{align*}
\Pi_{\mu, p}^{\tau} K+_p\Pi_{\mu, p}^{-\tau} K = \Pi_{\mu, p}^{+} K+_p\Pi_{\mu, p}^{-} K.
\end{align*}
This gives (\ref{zftzfa}).

Along the same lines, with (\ref{ftyt3}), (\ref{ftyt4}), (\ref{g2}) and (\ref{lpmz}), we can easily show that (\ref{zftzfb}) holds.
\end{proof}

Next, we prove the extreme inequality of the general $L_p$ $\mu$-projection body and the polar body of the general $L_p$ $\mu$-projection body.

\begin{theorem} \label{dliqw}
 Suppose that $\mu$ satisfies the homogeneity condition {\bf ({\rm I})} and $\omega$ is an even density function. If $K\in \mathscr{K}^n_o$, $p\geq 1$, and $\tau \in [-1,1]$, then
\begin{align}\label{the44}
\mu(\Pi_{\mu,p}^\ast K) \leq \mu(\Pi_{\mu,p}^{\tau,\ast} K) \leq \mu(\Pi_{\mu,p}^{\pm,\ast} K).
\end{align}
When $\tau \neq
0$, equality holds in the left inequality if and only if $\Pi_{\mu,p}^{\tau}K$ is origin-symmetric, and when $\tau \neq
\pm 1$, equality holds in the right inequality if and only if $\Pi_{\mu,p}^{\pm} K$ is origin-symmetric.
\end{theorem}
\begin{proof}
Clearly, if $\tau=\pm 1$ in (\ref{the44}), the inequality on the right of (\ref{the44}) is an identity. Let
$-1<\tau<1$. From (\ref{tytmz}) and (\ref{hrzh}), we get
\begin{equation}\label{maop}
\Pi_{\mu,p}^{\tau,*}K=g_1(\tau) \circ
\Pi_{\mu,p}^{+,*}K \widetilde{+}_{-p}\, g_2(\tau) \circ \Pi_{\mu,p}^{-,*}K.
\end{equation}
From Lemma \ref{rbmi}, we obtain
\begin{align*}
\mu(\Pi_{\mu,p}^{\tau,*}K)^{-sp}=\mu(g_1(\tau) \circ
\Pi_{\mu,p}^{+,*}K \widetilde{+}_{-p}\, g_2(\tau) \circ \Pi_{\mu,p}^{-,*}K)^{-sp} \geq g_1(\tau)\mu(\Pi_{\mu,p}^{+,*}K)^{-sp} + g_2(\tau) \mu(\Pi_{\mu,p}^{-,*}K)^{-sp},
\end{align*}
that is,
\begin{equation}\label{ineq1}
\mu(\Pi_{\mu,p}^{\tau,*}K)\leq \mu(\Pi_{\mu,p}^{\pm,*}K),
\end{equation}
 equality holds if and only if $\Pi_{\mu,p}^{+}K$ and $\Pi_{\mu,p}^{-}K$ are dilatates. This and (\ref{zxtds}) imply that the equality holds on the right of (\ref{the44}) if and only if $\Pi_{\mu,p}^{+}K$ and $\Pi_{\mu,p}^{-}K$ are origin-symmetric.

Next, let's prove the other part of this theorem. If $\tau= 0$, the inequality on the left of (\ref{the44}) is an identity. Let $\tau\neq 0$. By (\ref{RA}), (\ref{maop}) and (\ref{g2}), we deduce
\begin{align}\label{jxt1}
\rho(\Pi_{\mu,p}^{\tau,*}K,u)^{-p}=g_1(\tau) \rho(\Pi_{\mu,p}^{+,*}K,u)^{-p} +\, g_2(\tau) \rho(\Pi_{\mu,p}^{-,*}K,u)^{-p}
\end{align}
and 
\begin{align}\label{jxt2}
\rho(\Pi_{\mu,p}^{-\tau,*}K,u)^{-p}=g_2(\tau) \rho(\Pi_{\mu,p}^{+,*}K,u)^{-p} +\, g_1(\tau) \rho(\Pi_{\mu,p}^{-,*}K,u)^{-p}.
\end{align}
Add (\ref{jxt1}) and (\ref{jxt2}) together, use (\ref{g2}) again to get
\begin{align}
\frac{1}{2}\rho(\Pi_{\mu,p}^{-\tau,*}K,u)^{-p}+\frac{1}{2}\rho(\Pi_{\mu,p}^{\tau,*}K,u)^{-p}=\frac{1}{2}\rho(\Pi_{\mu,p}^{+,*}K,u)^{-p} +\, \frac{1}{2} \rho(\Pi_{\mu,p}^{-,*}K,u)^{-p},
\end{align}
combining the above formula with (\ref{RA}), (\ref{SRF}), (\ref{lpmz}), (\ref{tytzf}) and using (\ref{SRF}) again, we can deduce
\begin{align*}
\rho\left(\frac{1}{2}\circ\Pi_{\mu,p}^{-\tau,*}K \widetilde{+}_{-p}\,\frac{1}{2}\circ\Pi_{\mu,p}^{\tau,*}K,u\right)^{-p}=\rho(\Pi_{\mu,p}^{*}K,u)^{-p},
\end{align*}
hence,
\begin{align*}
\Pi_{\mu,p}^{*}K=\frac{1}{2}\circ\Pi_{\mu,p}^{-\tau,*}K \widetilde{+}_{-p}\, \frac{1}{2}\circ\Pi_{\mu,p}^{\tau,*}K.
\end{align*}
From this and Lemma \ref{rbmi}, we get
\begin{align*}
\mu(\Pi_{\mu,p}^{*}K)^{-sp}\geq\frac{1}{2}\mu(\Pi_{\mu,p}^{-\tau,*}K)^{-sp} +\frac{1}{2}\mu(\Pi_{\mu,p}^{\tau,*}K)^{-sp},
\end{align*}
by (\ref{pro1}), we obtain 
\begin{align*}
\mu(\Pi_{\mu,p}^{*}K)\leq\mu(\Pi_{\mu,p}^{\tau,*}K),
\end{align*}
with equality if and only if $\Pi_{\mu,p}^{\tau}K$ and $\Pi_{\mu,p}^{-\tau}K$ are dilatates. By (\ref{pro1}), we know that the equality holds on the left of the (\ref{the44}) if and only if $\Pi_{\mu,p}^{\tau}K$ is origin-symmetric.
\end{proof}

\begin{theorem}
 Suppose that $\mu$ satisfies the homogeneity condition {\bf ({\rm I})} and concavity condition {\bf ({\rm II})} and $\omega$ is an even density function. If $K\in \mathscr{K}^n_o$, $p\geq 1$, and $\tau \in [-1,1]$, then
\begin{align*}
\mu(\Pi_{\mu,p} K) \geq \mu(\Pi_{\mu,p}^\tau K) \geq \mu(\Pi_{\mu,p}^\pm K).
\end{align*}
\end{theorem}
\begin{proof}
From Lemma \ref{LEMBM}, (\ref{tytzf}), (\ref{tzfgx}) and (\ref{g2}), we get
\begin{align*}
\frac{1+(g_1(\tau)+ g_2(\tau))(p-1)}{p}\mu(\Pi_{\mu,p}^{\tau}K)^s&\geq g_1(\tau)\mu(\Pi_{\mu,p}^+K)^{s}+ g_2(\tau)\mu(\Pi_{\mu,p}^-K)^{s}\\
&=\mu(\Pi_{\mu,p}^\pm K)^{s},
\end{align*}
i.e.,
\begin{align*}
\mu(\Pi_{\mu,p}^{\tau}K)\geq\mu(\Pi_{\mu,p}^\pm K).
\end{align*}

In order to show that the inequality on the left holds, by (\ref{zftzfa}) and (\ref{tytzf}), it follows that
\begin{align*}
\Pi_{\mu, p}K=\frac{1}{2}\cdot\Pi_{\mu, p}^{\tau} K+_p\frac{1}{2}\cdot\Pi_{\mu, p}^{-\tau} K,
\end{align*}
from this, and using Lemma \ref{LEMBM}  and  (\ref{pro1}), we can immediately obtain
\begin{align*}
\mu(\Pi_{\mu,p}K)\geq\mu(\Pi_{\mu,p}^{\tau} K).
\end{align*}
This proves our desire inequality.
\end{proof}

The proof of the extreme inequality of the general $L_p$ $\mu$-centroid body is similar to that of the general $L_p$ $\mu$-projection body to a great extent. We only give the sketch of proof of the following theorems. 
\begin{theorem}\label{dliqw2}
 Suppose that $\mu$ satisfies the  homogeneity condition {\bf ({\rm I})} and $\omega$ is an even density function. If $K\in \mathscr{S}^n_o$, $p\geq 1$, and $\tau \in [-1,1]$, then
\begin{align}\label{the46}
\mu(\Gamma_{\mu,p}^\ast K) \leq \mu(\Gamma_{\mu,p}^{\tau,\ast} K) \leq \mu(\Gamma_{\mu,p}^{\pm,\ast} K).
\end{align}
When $\tau \neq
0$, equality holds in the left inequality if and only if $\Gamma_{\mu,p}^{\tau}K$ is origin-symmetric, and when $\tau \neq
\pm 1$, equality holds in the right inequality if and only if $\Gamma_{\mu,p}^{\pm} K$ is origin-symmetric.
\end{theorem}
\begin{proof}
Obviously, if $\tau=\pm 1$ in (\ref{the46}), the inequality on the right of (\ref{the46}) is an identity. Let
$-1<\tau<1$. From (\ref{mtaupm}) and (\ref{hrzh}), it shows that
\begin{equation}\label{maop2}
\Gamma_{\mu,p}^{\tau,*}K=g_1(\tau) \circ
\Gamma_{\mu,p}^{+,*}K \widetilde{+}_{-p}\, g_2(\tau) \circ \Gamma_{\mu,p}^{-,*}K,
\end{equation}
from this, with Lemma \ref{rbmi} and (\ref{hrzh}), we deduce
\begin{equation*}
\mu(\Gamma_{\mu,p}^{\pm,*}K)\geq \mu(\Gamma_{\mu,p}^{\tau,*}K),
\end{equation*}
 equality holds if and only if $\Gamma_{\mu,p}^{+}K$ and $\Gamma_{\mu,p}^{-}K$ are dilatates. This and (\ref{tzfgx}) imply that the equality holds on the right of the (\ref{the46}) if and only if $\Gamma_{\mu,p}^{+}K$ and $\Gamma_{\mu,p}^{-}K$ are origin-symmetric.

Next, let's prove the left part of (\ref{the46}). If $\tau= 0$, the inequality on the left of (\ref{the46}) is an identity. Let $\tau\neq 0$. By (\ref{RA}), (\ref{maop2}) and (\ref{g2}), we deduce
\begin{align}\label{jxt12}
\rho(\Gamma_{\mu,p}^{\tau,*}K,u)^{-p}=g_1(\tau) \rho(\Gamma_{\mu,p}^{+,*}K,u)^{-p} +\, g_2(\tau) \rho(\Gamma_{\mu,p}^{-,*}K,u)^{-p}
\end{align}
and 
\begin{align}\label{jxt22}
\rho(\Gamma_{\mu,p}^{-\tau,*}K,u)^{-p}=g_2(\tau) \rho(\Gamma_{\mu,p}^{+,*}K,u)^{-p} +\, g_1(\tau) \rho(\Gamma_{\mu,p}^{-,*}K,u)^{-p}.
\end{align}
By (\ref{jxt12}) and (\ref{jxt22}), useing (\ref{g2}) again, 
combining with (\ref{RA}), (\ref{SRF}), (\ref{lpmz}), (\ref{tytzfc}) and using (\ref{SRF}) again, we can deduce
\begin{align*}
\rho\left(\frac{1}{2}\circ\Gamma_{\mu,p}^{-\tau,*}K \widetilde{+}_{-p}\,\frac{1}{2}\circ\Gamma_{\mu,p}^{\tau,*}K,u\right)^{-p}&=\frac{1}{2}\rho(\Gamma_{\mu,p}^{+,*}K,u)^{-p} +\, \frac{1}{2} \rho(\Gamma_{\mu,p}^{-,*}K,u)^{-p}\\
&=\rho(\Gamma_{\mu,p}^{*}K,u)^{-p},
\end{align*}
therefore,
\begin{align*}
\Gamma_{\mu,p}^{*}K=\frac{1}{2}\circ\Gamma_{\mu,p}^{-\tau,*}K \widetilde{+}_{-p}\, \frac{1}{2}\circ\Gamma_{\mu,p}^{\tau,*}K.
\end{align*}
From this, with Lemma \ref{rbmi} and (\ref{pro2}), we deduce
\begin{align*}
\mu(\Gamma_{\mu,p}^{\tau,*}K)\leq \mu(\Gamma_{\mu,p}^{*}K).
\end{align*}
with equality if and only if $\Gamma_{\mu,p}^{\tau}K$ and $\Gamma_{\mu,p}^{-\tau}K$ are dilatates. By (\ref{pro2}), we know that the equality holds on the left of the (\ref{the46}) if and only if $\Gamma_{\mu,p}^{\tau}K$ is origin-symmetric.
\end{proof}

\begin{theorem}
 Suppose that $\mu$ satisfies the homogeneity condition {\bf ({\rm I})} and concavity condition {\bf ({\rm II})} and $\omega$ is an even density function. If $K\in \mathscr{S}^n_o$, $p\geq 1$ and $\tau \in [-1,1]$, then
\begin{align*}
\mu(\Gamma_{\mu,p} K) \geq \mu(\Gamma_{\mu,p}^\tau K) \geq \mu(\Gamma_{\mu,p}^\pm K).
\end{align*}
\end{theorem}
\begin{proof}
By Lemma \ref{LEMBM}, (\ref{mtaupm}), (\ref{zxtds}) and (\ref{g2}), we get
\begin{align*}
\mu(\Gamma_{\mu,p}^{\tau}K)\geq\mu(\Gamma_{\mu,p}^\pm K).
\end{align*}
To prove the other part of the inequality, by (\ref{zftzfb}) and (\ref{tytzfc}), it follows that
\begin{align*}
\Gamma_{\mu, p}K=\frac{1}{2}\cdot\Gamma_{\mu, p}^{\tau} K+_p\frac{1}{2}\cdot\Gamma_{\mu, p}^{-\tau} K,
\end{align*}
from this, and using Lemma \ref{LEMBM}, (\ref{pro2}) and (\ref{g2}), we obtain 
\begin{align*}
\mu(\Gamma_{\mu,p}K)\geq\mu(\Gamma_{\mu,p}^{\tau} K).
\end{align*}
This proves our desire inequality.
\end{proof}

\section{Brunn-Minkowski Type Inequalities and Monotone inequalities}\label{SEC5}

The $L_p$-Blaschke additions of measure is defined by Wu \cite{WDH1}. Suppose that $\mu$ satisfies the homogeneity condition {\bf ({\rm I})} and concavity condition {\bf ({\rm II})} and $\omega$ is an even density function. Let $p \geq 1$ and $p \neq 1/s$,  for $K, L \in \mathscr{K}_{o}^{n}$, the $L_{p}$-Blaschke addition of measures, $K \#_{\mu, p} L$  is a convex body, defined by
\begin{align}\label{glpba}
S_{\mu, p}\left(K \#_{\mu, p} L, \cdot\right)=S_{\mu, p}(K, \cdot)+S_{\mu, p}(L, \cdot).
\end{align}

Next, we prove Brunn-Minkowski type inequalities of the general $L_p$ $\mu$-projection body and its polar body with respect to $L_p$-Blaschke addition of measures.

\begin{theorem}\label{THM51}
 Suppose that $\mu$ satisfies the homogeneity condition {\bf ({\rm I})} and concavity condition {\bf ({\rm II})} and $\omega$ is an even density function. Let $p \geq 1$ and $p \neq 1/s$. If $K, L \in \mathscr{K}_{o}^{n}$ and $\tau \in [-1,1]$, then
\begin{align*}
(2-1/p)\mu(\Pi_{\mu,p}^{\tau}(K \#_{\mu, p} L))^s \geq \mu(\Pi_{\mu,p}^{\tau}K)^{s}+\mu(\Pi_{\mu,p}^{\tau}L)^{s}.
\end{align*}
\end{theorem}

\begin{proof}

From (\ref{glpba}) and (\ref{ghmt}), we have
\begin{align} \label{tyzchs}
h(\Pi_{\mu,p}^{\tau}(K \#_{\mu, p} L),u)^p=h(\Pi_{\mu,p}^{\tau}K,u)^p +h(\Pi_{\mu,p}^{\tau}L,u)^p, \qquad u \in \mathbb{S}^{n-1},
\end{align}
from this, (\ref{lphmt}) and (\ref{lphtm2}), for $N\in\mathscr{K}_{o}^{n}$, we have
\begin{align*}
& V_{\mu, p}(N, \Pi_{\mu,p}^{\tau}(K \#_{\mu, p} L))\\
=&s\int_{\mathbb{S}^{n-1}} h(\Pi_{\mu,p}^{\tau}(K \#_{\mu, p} L),u)^p d S_{\mu, p}(N, u)\\
=&V_{\mu, p}(N, \Pi_{\mu,p}^{\tau}K)+V_{\mu, p}(N, \Pi_{\mu,p}^{\tau}L)\\
\geq& p \mu(N)^{1-s} \mu(\Pi_{\mu,p}^{\tau}K)^{s}+(1-p) \mu(N)+p \mu(N)^{1-s} \mu(\Pi_{\mu,p}^{\tau}L)^{s}+(1-p) \mu(N).
\end{align*}
Let $N=\Pi_{\mu,p}^{\tau}(K \#_{\mu, p} L)$, then
\begin{align*}
(2-1/p)\mu(\Pi_{\mu,p}^{\tau}(K \#_{\mu, p} L))^s \geq \mu(\Pi_{\mu,p}^{\tau}K)^{s}+\mu(\Pi_{\mu,p}^{\tau}L)^{s}.
\end{align*}

\end{proof}

Particularly, if $\tau=0$ and $p=1$ in Theorem \ref{THM51}, then the following result is obvious.
\begin{corollary}
 Suppose that $\mu$ satisfies the homogeneity condition {\bf ({\rm I})} and concavity condition {\bf ({\rm II})} and $\omega$ is an even density function. If $K, L \in \mathscr{K}_{o}^{n}$, then
\begin{align*}
\mu(\Pi_{\mu}(K \#_{\mu} L))^s \geq \mu(\Pi_{\mu}K)^{s}+\mu(\Pi_{\mu}L)^{s}.
\end{align*}
\end{corollary}

\begin{theorem}\label{thm5.3}
 Suppose that $\mu$ satisfies the homogeneity condition {\bf ({\rm I})} and concavity condition {\bf ({\rm II})} and $\omega$ is an even density function. Let $p \geq 1$ and $p \neq 1/s$. If $K, L \in \mathscr{K}_{o}^{n}$ and $\tau \in [-1,1]$, then
\begin{align}\label{tytjbmi}
\mu(\Pi_{\mu,p}^{\tau,*}(K \#_{\mu, p} L))^{-sp} \geq \mu(\Pi_{\mu,p}^{\tau,*}K)^{-sp}+\mu(\Pi_{\mu,p}^{\tau,*}L)^{-sp}.
\end{align}
\end{theorem}

\begin{proof}

From (\ref{LpHT}), (\ref{SRF}), (\ref{tyzchs}), and Lemma \ref{DMMBS}, for $Q\in\mathscr{K}_{o}^{n}$, we get
\begin{align*}
\widetilde{V}_{\mu,-p}(Q, \Pi_{\mu,p}^{\tau,*}(K \#_{\mu, p} L))=&s\int_{\mathbb{S}^{n-1}}\rho(Q,u)^\frac{1+sp}{s}\rho(\Pi_{\mu,p}^{\tau,*}(K \#_{\mu, p} L),u)^{-p}d\mu(u)\\
=&s\int_{\mathbb{S}^{n-1}}\rho(Q,u)^\frac{1+sp}{s}h(\Pi_{\mu,p}^{\tau}(K \#_{\mu, p} L),u)^{p}d\mu(u)\\
=&s\int_{\mathbb{S}^{n-1}}\rho(Q,u)^\frac{1+sp}{s}(h(\Pi_{\mu,p}^{\tau}K,u)^p +h(\Pi_{\mu,p}^{\tau}L,u)^p)d\mu(u)\\
=&\widetilde{V}_{\mu,-p}(Q, \Pi_{\mu,p}^{\tau,*}K)+\widetilde{V}_{\mu,-p}(Q, \Pi_{\mu,p}^{\tau,*}L)\\
\geq&\mu(Q)^{1+sp}\mu(\Pi_{\mu,p}^{\tau,*}L)^{-sp}+\mu(Q)^{1+sp}\mu(\Pi_{\mu,p}^{\tau,*}K)^{-sp}
\end{align*}
Let $Q=\Pi_{\mu,p}^{\tau,*}(K \#_{\mu, p} L)$, we obtain 
\begin{align*}
\mu(\Pi_{\mu,p}^{\tau,*}(K \#_{\mu, p} L))^{-sp}
\geq\mu(\Pi_{\mu,p}^{\tau,*}L)^{-sp}+\mu(\Pi_{\mu,p}^{\tau,*}K)^{-sp}.
\end{align*}
\end{proof}

Clearly, if $\tau=0$ and $p=1$ in Theorem \ref{thm5.3}, then we can obtain

\begin{corollary}
 Suppose that $\mu$ satisfies the homogeneity condition {\bf ({\rm I})} and concavity condition {\bf ({\rm II})} and $\omega$ is an even density function. If $K, L \in \mathscr{K}_{o}^{n}$, then
\begin{align*}
\mu(\Pi_{\mu}^{\ast}(K \#_{\mu} L))^{-s} \geq \mu(\Pi_{\mu}^{\ast}K)^{-s}+\mu(\Pi_{\mu}^{\ast}L)^{-s}.
\end{align*}
\end{corollary}

Finally, we prove monotone inequalities for the general $L_p$ $\mu$-centroid body and its polar.
\begin{theorem}\label{thm5.5}
(a) Let $\mu$ satisfy the homogeneity condition {\bf ({\rm I})} and concavity condition {\bf ({\rm II})}, if $ p\geq 1$ and $\tau \in [-1,1]$, $K,L \in \mathscr{S}^n_{\mathrm{o}}$ and $K \subseteq L$, then 
\begin{align}\label{eq51}
p\mu(K)\mu(\Gamma_{\mu,p}^\tau K)^s\leq (\mu(L)+(p-1)\mu(K))\mu(\Gamma_{\mu,p}^\tau L)^s.
\end{align}
(b) Let $\mu$ satisfy the homogeneity condition {\bf ({\rm I})}, if $ p\geq 1$ and $\tau \in [-1,1]$, $K,L \in \mathscr{S}^n_{\mathrm{o}}$ and $K \subseteq L$, then
\begin{align}\label{eq52}
\frac{\mu(\Gamma_{\mu,p}^{\tau,\ast} K)^{sp}}{\mu(K)}\geq \frac{\mu(\Gamma_{\mu,p}^{\tau,\ast} L)^{sp}}{\mu(L)},
\end{align}
equality holds in (\ref{eq52}) if and only if $K=L$.
\end{theorem}

\begin{proof}
Since $K,L\in\mathscr{S}^n_o$ and $K\subseteq L$, by (\ref{RC}), (\ref{LpHT}), for $Q\in\mathscr{K}^n_o$, we obtain
\begin{align}\label{dlptjb}
\widetilde{V}_{\mu,-p}(K,Q)\leq \widetilde{V}_{\mu,-p}(L,Q),\end{align}
 equality holds if and only if $K=L$.

For (a), taking $Q=\Pi_{\mu,p}^{\tau,\ast}M $ and $M\in\mathscr{K}^n_o$ in (\ref{dlptjb}), we obtain
\begin{align*}
\widetilde{V}_{\mu,-p}(K,\Pi_{\mu,p}^{\tau,\ast}M)\leq \widetilde{V}_{\mu,-p}
(L,\Pi_{\mu,p}^{\tau,\ast}M),
\end{align*}
From this and Proposition \ref{durch}, we get
\begin{align*}
\mu(K)V_{\mu,p}(M,\Gamma_{\mu,p}^\tau K)\leq \mu(L)V_{\mu,p}(M,\Gamma_{\mu,p}^{\tau} L),
\end{align*}
in the above inequality, let $M=\Gamma_{\mu,p}^\tau L$, and from (\ref{lphtm2}), we have
\begin{align*}
\mu(L)\mu(\Gamma_{\mu,p}^\tau L)&\geq \mu(K)V_{\mu,p}(\Gamma_{\mu,p}^\tau L,\Gamma_{\mu,p}^\tau K)\\
&\geq p \mu(K)(\mu(\Gamma_{\mu,p}^\tau L)^{1-s} \mu(\Gamma_{\mu,p}^\tau K)^{s}+(1-p) \mu(\Gamma_{\mu,p}^\tau L)).
\end{align*}
This immediately gives inequality (\ref{eq51}).

For (b),  let $Q=\Gamma_{\mu,p}^{\tau,\ast} N$ and $N\in\mathscr{K}^n_o$ in (\ref{dlptjb}), we obtain
\begin{align*}
\widetilde{V}_{\mu,-p}(K,\Gamma_{\mu,p}^{\tau,\ast} N)\leq \widetilde{V}_{\mu,-p}(L,\Gamma_{\mu,p}^{\tau,\ast} N),
\end{align*}
with equality if and only if $K=L$. From this and using Proposition \ref{prop3}, we obtain
\begin{align}\label{mgarm}
\mu(K) \widetilde{V}_{\mu,-p}(N,\Gamma_{\mu,p}^{\tau,\ast} K)\leq\mu(L) \widetilde{V}_{\mu,-p}(N,\Gamma_{\mu,p}^{\tau,\ast} L).
\end{align}
Taking $N=\Gamma_{\mu,p}^{\tau,\ast} L$ in (\ref{mgarm}), and using Lemma \ref{DMMBS}, we deduce
\begin{align}\label{eq53}
\nonumber\mu(L)\mu(\Gamma_{\mu,p}^{\tau,\ast} L)\geq& \mu(K) \widetilde{V}_{\mu,-p}(\Gamma_{\mu,p}^{\tau,\ast} L,\Gamma_{\mu,p}^{\tau,\ast} K)\\
\geq& \mu(K)\mu(\Gamma_{\mu,p}^{\tau,\ast} L)^{1+sp}\mu(\Gamma_{\mu,p}^{\tau,\ast} K)^{-sp}.
\end{align}
with equality if and only if $\Gamma_{\mu,p}^\tau K$ and $\Gamma_{\mu,p}^\tau L$ are dilatates.

(\ref{eq53}) yields inequality (\ref{eq52}). By the equality conditions of inequalities (\ref{dlptjb}) and (\ref{eq53}), we see that equality holds in (\ref{eq52}) if and only if $K=L$.
\end{proof}

In particular, if $\tau=0$ and $p=1$ in Theorem \ref{thm5.5}, then we obtain the following

\begin{corollary}
(a) Let $\mu$ satisfy the homogeneity condition {\bf ({\rm I})} and concavity condition {\bf ({\rm II})}, if  $K,L \in \mathscr{S}^n_{\mathrm{o}}$ and $K \subseteq L$, then 
\begin{align*}
\mu(K)\mu(\Gamma_{\mu} K)^s\leq \mu(L)\mu(\Gamma_{\mu} L)^s.
\end{align*}
(b) Let $\mu$ satisfy the  homogeneity condition {\bf ({\rm I})}, if  $K,L \in \mathscr{S}^n_{\mathrm{o}}$ and $K \subseteq L$, then 
\begin{align*}
\frac{\mu(\Gamma_{\mu}^{\ast} K)^{s}}{\mu(K)}\geq \frac{\mu(\Gamma_{\mu}^{\ast} L)^{s}}{\mu(L)},
\end{align*}
equality holds if and only if $K=L$.
\end{corollary}

\end{sloppypar}
\end{document}